\documentclass[12pt,leqno]{amsproc}
\usepackage{graphicx}
\usepackage{amscd}
\usepackage{amsmath}
\usepackage{amsfonts}
\usepackage{amssymb}
\theoremstyle{plain}

\newtheorem{cor}{Corollary}

\newtheorem{dfn}{Definition}

\newtheorem{lem}{Lemma}

\newtheorem{prop}{Proposition}

\newtheorem{thm}{Theorem}
\numberwithin{equation}{section}

\theoremstyle{definition}

\newtheorem{rmk}{Remark}

\newcommand{\lra}{{\longrightarrow}}

\newcommand{\II}{\begin{enumerate}}
\newcommand{\III}{\end{enumerate}}

\begin{document}

\title[Existence of well-filterifications of $T_0$ topological spaces]
{Existence of well-filterifications of $T_0$ topological spaces}

\author[G. Wu]{Wu Guohua}
\address{Division of Mathematical Sciences, School of Physical and
  	Mathematical Sciences\\
  	Nanyang Technological University\\
  	Singapore}
\email{guohua.wu@ntu.edu.sg}

\author[X. Xi]{Xi Xiaoyong}
\address{Division of Mathematical Sciences, School of Physical and
  	Mathematical Sciences\\
  	Nanyang Technological University\\
  	Singapore}
\email{xiaoyong.xi@ntu.edu.sg}

\author[X. Xu]{Xu Xiaoquan}
\address{School of Mathematics and Statistics\\
Minnan Normal University\\Fujian, Zhangzhou 363000, P.R. China}
 \email{xiqxu2002@163.com}

\author[D. Zhao]{Zhao Dongsheng}
\address{Mathematics and Mathematics Education\\
National Institute of Education\\
Nanyang Technological University, Singapore}
\email{dongsheng.zhao@nie.edu.sg}


\begin{abstract} We prove that for every $T_0$ space $X$, there is a well-filtered space $W(X)$ and a continuous mapping $\eta_X: X\lra W(X)$, such that
for any well-filtered space $Y$ and any continuous mapping $f: X\lra Y$ there is a unique continuous mapping $\hat{f}: W(X)\lra Y$ such that $f=\hat{f}\circ \eta_X$. Such a space $W(X)$ will be called the well-filterification of $X$. This result gives a positive answer to one of the major open problems on well-filtered spaces. Another result on well-filtered spaces we will prove is that the product of two well-filtered spaces is well-filtered.
\end{abstract}

\maketitle

The three important topological properties for non-Hausdorff spaces are the sobriety, monotone convergence (or being d-space) and well-filteredness. It  has  been proved by different authors that every $T_0$ space has a sobrification and a d-completion \cite{comp}\cite{ersov-1999c}\cite{keimel-lawson}, equivalently, the subcategory of all sober spaces and that of all monotone convergent spaces are reflexive in the category of all $T_0$ spaces. However, it is still unknown whether the category of all well-filtered spaces is reflexive in the category of all $T_0$ spaces.
In this paper we give a positive answer to this problem. Our main strategy is to use the criteria for the existence of K-fication of $T_0$ spaces  suggested by Keimel and Lawson in \cite{keimel-lawson}. Another problem on well-filteredeness is whether the product of two well-filtered spaces is a well-filtered space.  We will also give a positive answer to this problem.

\section{K-fication}

Assume that a topological property, called K-property, is given. By \cite{keimel-lawson}, a K-fication of a $T_0$ space $X$ is a space $F(X)$ with K-property and a continuous mapping $\eta_X: X\lra F(X)$ which is universal among all continuous mappings from $X$ to  spaces with K-property: for any continuous mapping $g: X\lra Z$ where $Z$ is a space  with K-property, there is a unique continuous mapping $\hat{g}: F(X)\lra Z$ such that $g=\hat{g}\circ \eta_X$. If every $T_0$ space has a K-fication, then the category of all $T_0$ spaces with K-property is reflexive in the category of all $T_0$ spaces.

By Keimel and Lawson \cite{keimel-lawson}, if the K-property satisfies the following four conditions, then every $T_0$ space has a K-fication:

(K1) Every sober space has K-property.

(K2) If $X$ has K-property and $Y$ is homeomorphic to $X$, then $Y$ also has K-property.

(K3) If $\{X_i\}_{i\in I}$ is a family of subspaces of a sober space such that each $X_i$  has K-property, then the subspace $\bigcap_{i\in I}X_i$ also has K-property;

(K4) If $f: X\lra Y$ is a continuous mapping between sober spaces  $X$ and $Y$, then for any subspace $Z$ of $Y$ with K-property, $f^{-1}(Z)$  has K-property.

\vskip 0.5cm
For a $T_0$ space $(X, \tau)$, the specialization order on $X$, written $\le_{\tau}$ (or just $\le$), is define as $x\le_{\tau}y$ iff $x\in cl(\{y\})$, where $cl$ is
the  closure operator.

As in a general poset we shall use the following standard notations for any subset $A$ of a $T_0$ space $(X, \tau)$:

$$\uparrow_X\!A=\{y\in X: \mbox{there is an } x\in A, x\le_{\tau} y\}.$$

$$\downarrow_X\!\!A=\{y\in X: \mbox{there is an } x\in A, y\le_{\tau} x\}.$$

\smallskip

For any $x\in X$, $\uparrow_X\!x=\uparrow_X\!\{x\}$ and $\downarrow_X\!x=\downarrow_X\!\{x\}$.

\smallskip
The symbol $\uparrow_X\!A$ $(\downarrow_X\!\!A, \uparrow_X\!x, \downarrow_X\!x, resp.)$ will be simply written as
 $\uparrow\!\!A$ $(\downarrow\!\!A, \uparrow\!x, \downarrow\!x, resp.)$ if no ambiguous occurs.

By the definition of the specialization order on  $T_0$ space $X$, for any $x\in X$, $\downarrow\!\!x=cl(\{x\})$, the closure of $\{x\}$. Hence $\downarrow\!\!x$ is a closed set.

\begin{rmk}\label{upset in subspaces}
Let $X_1$ be a subspace of a $T_0$ space $(X, \tau)$. For any $x\in X_1$, the closure $cl_{X_1}\{x\}$ of $\{x\}$ in $X_1$ equals $X_1\cap cl_{X}\{x\}$, where $cl_X\{x\}$ is the closure of $\{x\}$ in $X$. Hence for any $x, y\in X_1$, $x\leq y$ holds in $X_1$ if and only if $x\leq y$ holds in $X$. In other words, the specialization order on $X_1$ is the restriction of $\leq_{\tau}$ on $X_1$.

Hence for any $x\in X_1$,  we have $\downarrow_{X_1}\!\!x=(\downarrow_X\!\!x)\cap X_1$ ($\uparrow_{X_1}\!\!x=(\uparrow_X\!\!x)\cap X_1$, resp.).

In general, for any $A\subseteq X_1$,  $\downarrow_{X_1}\!\!A=(\downarrow_X\!\!A)\cap X_1$ ($\uparrow_{X_1}\!\!A=(\uparrow_X\!\!A)\cap X_1$, resp.).

\end{rmk}

A subset $A$ of a space $X$ is {\sl irreducible} if for any closed sets $F_1, F_2$ of $X$, $A\subseteq F_1\cup F_2$ implies $A\subseteq F_1$ or $A\subseteq F_2$.
A $T_0$ space $X$ is called sober if for any nonempty irreducible closed set $F$, $F=cl(\{x\})$ for some $x\in X$.

The sobriety satisfies all conditions (K1)-(K4), hence the so-called soberification exists for each $T_0$ space \cite{comp}.

For any subset $A$ of a space $X$, the saturation of $A$, denoted by $sat(A)$, is defined to be
$$ sat(A)=\bigcap\{U\in\mathcal{O}(X): A\subseteq U\},$$
where $\mathcal{O}(X)$ is the set of all open sets of $X$.

A subset $A$ of a space $X$ is called saturated if $A=sat(A)$.

The saturation of any subset is a saturated set, and the saturation of every compact set is compact\cite{comp}\cite{Jean-2013}.

It is a standard fact that for any subset $A$ of a space $X$\cite{comp}\cite{Jean-2013},
$$ sat(A)=\uparrow\!\!A.$$

\begin{dfn}
A $T_0$  space $X$ is called {\sl  well-filtered}  if for any open set $U$ and any filtered family  $\mathcal{F}$ of saturated compact subsets of $X$ (for any $F_1, F_2\in\mathcal{F}$, there exists $F_3\in \mathcal{F}$ such that $F_3\subseteq F_1\cap F_2$), $\bigcap\mathcal{F}\subseteq U$ implies $F\subseteq U$ for some $F\in\mathcal{F}$.
\end{dfn}

\begin{rmk}\label{remark on sober spaces}

(1) Every sober space is well-filtered, and a locally compact space is sober iff it is well-filtered \cite{comp}\cite{Jean-2013}.

(2) A $T_0$  space $X$ is called a  monotone convergent space (or d-space), if for any directed subset $D$ of $X$ (with respect to the specialization order on $X$), $\bigvee\,D$ exists and $D$ converges (as a net ) to $\bigvee\,D$.  Every well-filtered space is  monotone convergent. The monotone convergence is a topological property satisfying all conditions (K1)-(K4), thus the d-completion exists for each $T_0$ spaces \cite{keimel-lawson}.

\end{rmk}

In this paper we prove that the well-filtered property satisfies all the conditions (K1)-(K4), hence the well-filterification exists for every $T_0$ space.

\begin{rmk} \label{remarks on well-filtered spaces}

(1) If a space $X$ is well-filtered and $\{F_i\}_{i\in I}$ is a filtered family of (non-empty) saturated compact sets, then $\bigcap\{F_i: i\in I\}$ is a (non-empty) saturated compact set \cite{Jean-2013}\cite{xi-zhao-MSCS-well-filtered}.

(2) For any saturated compact set $E$ in a $T_0$ space, $E=\uparrow\!C$, where $C$ is a compact set and an anti-chain (with respect to the specialization order). In other words, every element in $E$ is above some minimal element(s) of $E$. This claim follows from the compactness of $E$ and the Maximal Chain Principle.

\end{rmk}

For more about sober spaces, well-filtered spaces, d-spaces  and saturated sets, we refer the reader to \cite{comp}\cite{Jean-2013}\cite{Klause-Heckmann}\cite{jia-yung-2016}\cite{Xi-Lawson-2017}\cite{xi-zhao-MSCS-well-filtered}.

\section{Existence of well-filterification}

We now verify that the well-filteredness satisfies all the conditions (K1)-(K4) given in \cite{keimel-lawson}. The condition (K1) holds due to the fact that every sober space is well-filtered (see  Remark \ref{remark on sober spaces} (1)). The condition (K2) holds because the well-filteredness is a topological property.
Thus we only need to verify the condition (K3) and (K4).

In what follows, all topological spaces to be considered are assumed to be $T_0$.

\begin{rmk}\label{fund remark}
If $A_1, A_2$ and  $A_3$ are subsets of a space $X$ such that $\uparrow\!\!A_3\subseteq \uparrow\!\!A_1\cap \uparrow\!\!A_2$, then for any lower set $F\subseteq X$ (i.e. $F=\downarrow\!\!F$), $$\uparrow\!\!(F\cap A_3)\subseteq \uparrow\!\!(F\cap A_1)\cap \uparrow\!\!(F\cap A_2).$$
In fact, let $y\in F\cap A_3$. Then $k_1\le y$ for some $k_1\in A_1$, and $k_2\le y$ for some $k_2\in A_2$.
Since $y\in F=\downarrow\!F$, we have that $k_1, k_2\in F$. It follows that $k_1\in F\cap A_1, k_2\in F\cap A_2$.
Hence $y\in \uparrow\!(F\cap\,A_1)\cap \uparrow\!(F\cap\,A_2).$ Therefore $ F\cap A_3\subseteq \uparrow\!(F\cap A_1)$ and $F\cap A_3\subseteq \uparrow\!(F\cap\,A_2),$ which then imply that $\uparrow\!\!(F\cap A_3)\subseteq \uparrow\!\!(F\cap A_1)$ and $\uparrow\!\!(F\cap A_3)\subseteq\uparrow\!\!(F\cap A_2)$, or
$$\uparrow\!\!(F\cap A_3)\subseteq \uparrow\!\!(F\cap A_1)\cap\uparrow\!\!(F\cap A_2).$$
\end{rmk}

\begin{rmk}\label{immage of filtered families}
Let $f: X\lra Y$ be a continuous mapping  between topological spaces.

(1) For any subset $A\subseteq X$, $f(\uparrow_X\!\!A)\subseteq \uparrow_Y\!\!f(A)$.

(2) If $A, B, C\subseteq X$ and $A\subseteq \uparrow_X\!\! B\cap\uparrow_X\!\!C$, then
$$\uparrow_Y\!\!f(A)\subseteq \uparrow_Y\!\!f(B)\cap\uparrow_Y\!\!f(C).$$

(3) If $\{\uparrow_X\!\!H_i\}_{i\in I}$ is a filtered family of subsets of $X$, then
$$\{\uparrow_Y\!\!f(H_i): i\in I\}$$
is a filtered family of subsets of $Y$.

\end{rmk}

The lemma below will be used to prove several other results.

\begin{lem}\label{fund lemma}
Let $X$ be a well-filtered space. Then for any filtered family $\{K_i\}_{i\in I}$ of nonempty compact saturated subsets of $X$, we have

(1) $\bigcap_{i\in I}K_i=\uparrow\!C$, where $C$ is a nonempty anti-chain;

(2) for each $a\in C$, $\bigcap_{i\in I}\uparrow\!(\downarrow\!a\cap K_i)=\uparrow\! a$.
\end{lem}

\begin{proof}

(1) This follows from (1)(2) of Remark \ref{remarks on well-filtered spaces}.

(2) For each $a\in C$, $a\in C\subseteq \uparrow\! C=\bigcap_{i\in I}K_i$, so $a\in K_i$ for each $i\in I$. In particular, $\downarrow\!a\cap K_i\not=\emptyset$ for each $i\in I$.

Now, by Remark \ref{fund remark} and that $\downarrow\!\!a$ $(=cl(\{a\}))$ is closed,  $\{\uparrow\!(\downarrow\!a\cap K_i)\}_{i\in I}$ is a filtered family of nonempty compact saturated sets. Since $X$ is well-filtered, applying  (1) to this new family of saturated compact sets,  there is a nonempty anti-chain $\hat{C}$ such that
$$\bigcap_{i\in I}\uparrow\!(\downarrow\!a\cap K_i)=\uparrow\!\hat{C}.$$

Note that $a\in \downarrow\!a\cap K_i (i\in I)$, thus
$$a\in\bigcap_{i\in I}\uparrow\!(\downarrow\!a\cap K_i)=\uparrow\!\hat{C}.$$
Hence $\downarrow\!a\cap \hat{C}\not=\emptyset$. Take $t\in \downarrow\!a\cap \hat{C}$. Then
$$t\in\hat{C}\subseteq \uparrow\!\hat{C}=\bigcap_{i\in I}\uparrow\!(\downarrow\!a\cap K_i)\subseteq \bigcap_{i\in I}\uparrow\!\!K_i=\bigcap_{i\in I}K_i=\uparrow\!C.$$

So there is $c\in C$ such that $c\le t$. Since $t\le a$, we have $c\le a$, implying $a=c$ because $C$ is an anti-chain and $a, c\in C$. Hence $a=c=t\in \hat{C}$.

We now show that $\hat{C}=\{a\}$. Assume, on the contrary that there is $a'\in \hat{C}-\{a\}$. Then, as $\hat{C}$ is an anti-chain, we have
$$\uparrow\!\hat{C}\subseteq (X-\downarrow\!a)\cup (X-\downarrow\!a'), $$
here $(X-\downarrow\!a)\cup (X-\downarrow\!a')=(X-cl\{a\})\cup (X-cl\{a'\})$ is an open set.

Since $X$ is well-filtered, $\{\uparrow\!\!(\downarrow\!a\cap K_i): i\in I\}$ is a filtered family of compact saturated sets of $X$ and
$$ \bigcap_{i\in I}\uparrow\!(\downarrow\!a\cap K_i)=\uparrow\!\hat{C}\subseteq (X-\downarrow\!a)\cup (X-\downarrow\!a'),$$
there is $i_0\in I$ such that $\downarrow\!a\cap K_{i_0}\subseteq (X-\downarrow\!a)\cup (X-\downarrow\!a').$
Since $\downarrow\!a\cap K_{i_0}$ and $X-\downarrow\!a$ are disjoint, $\downarrow\!a\cap K_{i_0}\subseteq X-\downarrow\!a'$, implying $(\downarrow\!a\cap K_{i_0})\cap \downarrow\!a'=\emptyset$. However $a'\in \uparrow\!(\downarrow\!a\cap K_{i_0})$, so $\downarrow\!a'\cap (\downarrow\!a\cap K_{i_0})\not=\emptyset$. This contradiction shows that $\hat{C}=\{a\}$, thus

$$\bigcap_{i\in I}\uparrow\!(\downarrow\!a\cap K_i)=\uparrow\!\hat{C}=\uparrow\!a,$$
as desired.

\end{proof}

\begin{rmk}\label{inclusion of saturation}

(1) For any open set $U$ of a space $X$, $U=sat(U)=\uparrow\!\!U$, and for any closed set $F$ of $X$, $F=\downarrow\!\!F$ (see \cite{comp}\cite{Jean-2013}).

(2) Let $A$ and $B$  be subsets of a space $X$. Then $sat(A) \subseteq sat(B)$ if and only if every open neighbourhood $U$ of $B$ (i.e. $B\subseteq U$) contains $A$.
\end{rmk}

The following result is a direct corollary of the general Topological Rudin's Lemma given by Keimel and Heckmann in \cite{Klause-Heckmann} (see Lemma 3. 1 of \cite{Klause-Heckmann}).

\begin{lem}\label{rutin'slemma}
Let $X$ be a topological space and $\mathcal{F}$ a filtered family of compact subsets of $X$ (for any $F_1, F_2\in\mathcal{F}$, there is $F\in\mathcal{F}$ such that
$F\subseteq \uparrow\!\!F_1\cap \uparrow\!\!F_2$). Any closed set $C \subseteq X$  that
meets all members of $\mathcal{F}$  contains an irreducible closed subset $A$ that still meets all
members of $\mathcal{F}$.

In addition, this irreducible closed set $A$ can be taken as a minimal one: if $A'\subseteq A$ is a proper closed subset of $A$, then $A'\cap F=\emptyset$ for some $F\in\mathcal{F}$.
\end{lem}

\begin{lem}\label{subspace}
Let $W$ be a subspace of a sober space $X$. Assume that  $\{K_i\}_{i\in I}$ is a filtered family of compact saturated subsets of $W$ and $U$ is an open set of $X$ such that (i) $\bigcap_{i\in I}K_i\subseteq U$ and (ii) $K_i\not\subseteq U (\forall i\in I)$. Then there is $e\in (X-W)\cap (X-U)$ such that
$$\bigcap_{i\in I}\uparrow\!(\downarrow\!e\cap K_i)=\uparrow\! e.$$
\end{lem}
\begin{proof} Note that each $K_i$ is also a compact subset of $X$. In addition, as $\{K_i: i\in I\}$ is a filtered family of compacts sets in the subspace $W$, it is also a filtered family of compact sets in $X$.

Now the closed set $U^c=X-U$ has a nonempty intersection with each $K_i (i\in I)$. By Lemma \ref{rutin'slemma}, there is a
minimal irreducible closed set $F\subseteq U^c$ such that $F\cap K_i\not=\emptyset$ ($i\in I)$. Since $X$ is sober, $F=cl\{e\}=\downarrow\!e$ for some $e\in X$.

Claim 1. $e\not\in W$.
As a mater of fact, if $e\in W$, then as $\downarrow\!e\cap K_i\not=\emptyset$, there is $k\in \downarrow\!e\cap K_i$ such that $k\le e$ holds in $X$. Therefore $k\le e$ holds in the subspace $W$ as well by Remark \ref{upset in subspaces}. Hence $e\in \uparrow_W\!\!K_i=K_i$ because $K_i$  is saturated in $W$.  It follows that  $e\in K_i$ for each $i\in I$. Then
$$e\in\bigcap_{i\in I}K_i\subseteq U,$$
which contradicts the assumption that $e\in \downarrow\!e=F\subseteq U^c$.

Claim 2. $\bigcap_{i\in I}\uparrow\!(\downarrow\!e\cap K_i)=\uparrow\!e.$

Clearly $\uparrow\!e\subseteq \bigcap_{i\in I}\uparrow\!(\downarrow\!e\cap K_i)$ holds. Note that, as an intersection of saturated sets $\uparrow\!(\downarrow\!e\cap K_i)$($i\in I$), $\bigcap_{i\in I}\uparrow\!(\downarrow\!e\cap K_i)$ is a saturated set. Hence, by (2) of Remark \ref{inclusion of saturation}, in order to show that  $\bigcap_{i\in I}\uparrow\!(\downarrow\!e\cap K_i)\subseteq \uparrow\!e$ holds, we only need to verify that every open neighbourhood of $e$ must contain  $\bigcap_{i\in I}\uparrow\!(\downarrow\!e\cap K_i)$ (note that $\uparrow\!\!e=sat(\{e\})$).

Let $V$ be any open set of $X$ containing $e$. As $\downarrow\!\!e$ is a closed set, each $\downarrow\!\!e\cap K_i$ is compact (the intersection of a compact set and a closed set is compact). In addition, by Remark \ref{fund remark}, we have that $\{\downarrow\!e\cap K_i: i\in I\}$ is a filtered family of sets in $X$.

If $V^c\cap\downarrow\!e\cap K_i\not=\emptyset $ for all $i\in I$, then by Lemma \ref{rutin'slemma}, there is a minimal irreducible closed set $G$ of $X$ such that $G\subseteq V^c$ and $G\cap\downarrow\!e\cap K_i\not=\emptyset $ for all $i\in I$. Then $G=\downarrow\!e'$ for some $e'\in X$ ($e'\in G\subseteq V^c$) because $X$ is sober. Now $\downarrow\!e'\cap\downarrow\!e\cap K_i\not=\emptyset $ for all $i\in I$, so $\downarrow\!e'\cap\downarrow\!e=\downarrow\!e$ due to the minimality of $\downarrow\!e$.

On the other hand, $(\downarrow\!e'\cap\downarrow\!e)\cap \downarrow\!e\cap K_i\not=\emptyset $ for all $i\in I$, so $\downarrow\!e'\cap\downarrow\!e=\downarrow\!e'$ due to the minimality of $\downarrow\!e'$. It thus follows that $e=e'$. But $e'\not\in V$ and $e\in V$, this contradiction shows that there is $i\in I$ such that $\downarrow\!e\cap K_i\subseteq V$, hence $\bigcap_{i\in I}\uparrow\!(\downarrow\!e\cap K_i)\subseteq V$. All these then show that $\bigcap_{i\in I}\uparrow\!(\downarrow\!e\cap K_i)\subseteq \uparrow\!e$.

Therefore $\bigcap_{i\in I}\uparrow(\downarrow\!e\cap K_i)= \uparrow\!e$.

The combination of Claim 1 and Claim 2 completes the proof.
\end{proof}

Now we prove that the well-filteredness satisfies condition (K4).
\begin{lem}
Let $f: (X, \tau) \lra (Y, \mu)$ be a continuous mapping between sober spaces. Then for any well-filtered subspace $Z$ of $Y$, $f^{-1}(Z)$ is a well-filtered subspace of $X$.
\end{lem}
\begin{proof}
Let $\{K_i\}_{i\in I}\subseteq f^{-1}(Z)$ be a filtered family of compact saturated subsets of $f^{-1}(Z)$ and $\bigcap\{K_i: i\in I\}\subseteq U$ with $U$ an open set of $X$. We show that  $K_i\subseteq U$ holds for some $i\in I$.

Assume that $K_i \not\subseteq U$ for every $i\in I$. Then by Lemma \ref{subspace}, there is $e\in (X-f^{-1}(Z))\cap (X-U)$
such that

$$\bigcap_{i\in I}\uparrow\!\!(\downarrow\!\!e\cap K_i)=\uparrow\!\! e.$$

Now we verify that the following equation holds in $Y$:
$$\bigcap_{i\in I}\uparrow_Y\!\! f(\downarrow\!\!e\cap K_i)=\uparrow_Y\!\! f(e).$$

Note that every continuous mapping preserves the specialization order. Hence it follows from (2) of Remark \ref{immage of filtered families} easily that
$$\bigcap_{i\in I}\uparrow_Y\!f(\downarrow\!e\cap K_i)\supseteq\uparrow_Y\! f(e).$$

Now let $V\subseteq Y$ be open and $f(e)\in V$. Then $e\in f^{-1}(V)$, and
$$\bigcap_{i\in I}\uparrow\!(\downarrow\!e\cap K_i)=\uparrow\! e\subseteq f^{-1}(V).$$
As $X$ is well-filtered (every sober space is well-filtered), there exists $i_0\in I$ such that
$$ \downarrow\!e\cap K_{i_0}\subseteq f^{-1}(V),$$
implying $f(\downarrow\!e\cap K_{i_0})\subseteq V.$ So
 $$\bigcap_{i\in I}\uparrow_Y\! f(\downarrow\!e\cap K_i)\subseteq \uparrow_Y\! f(\downarrow\!e\cap K_{i_0})\subseteq\uparrow_Y\!\! V=V.$$

By (2) of Remark \ref{inclusion of saturation},
$$\bigcap_{i\in I}\uparrow_Y\! f(\downarrow\!e\cap K_i)\subseteq \uparrow_Y\!f(e).$$

Therefore
$$\bigcap_{i\in I}\uparrow_Y\! f(\downarrow\!e\cap K_i)= \uparrow_Y\!f(e).$$

Since $e\in X - f^{-1}(Z)$, $f(e)\not\in Z$ and hence
$$\bigcap_{i\in I}(\uparrow_Y\! f(\downarrow\!e\cap K_i))\cap Z= \uparrow_Y\!f(e)\cap Z\subseteq Z-\downarrow_Y\!f(e).$$

By Remark \ref{immage of filtered families}, $\{\uparrow_Y\! (f(\downarrow\!e\cap K_i): i\in I\}$ is a filtered family of subsets of $Y$, then
$\{\uparrow_Y\! f(\downarrow\!e\cap K_i))\cap Z: i\in I\}$ is a filtered family of saturated compact subsets of the subspace $Z$. As $Z$ is well-filtered, there is $i_0\in I$ such that
$$\uparrow_Y\!(f(\downarrow\!e\cap K_{i_0}))\cap Z\subseteq Z-\downarrow_Y\!f(e).$$
But this is impossible. In fact, choose one $u\in \downarrow\!e\cap K_{i_0}$. Then $u\in K_{i_0}\subseteq f^{-1}(Z)$, implying $f(u)\in \uparrow_Y\!(f(\downarrow\!e\cap K_{i_0}))\cap Z$. On the other hand,
$u\le e$ implies $f(u)\le f(e)$, so $f(u)\in \downarrow_Y\!\!f(e)$, thus $f(u)\not\in Z-\downarrow_Y\!f(e)$.

This contradiction shows that there exists $i\in I$ such that $K_i\subseteq U$. Hence $f^{-1}(Z)$ is well-filtered.

\end{proof}

Next we verify that the well-filteredness satisfies condition (K3).
\begin{lem}\label{intersections of wf subspaces}
Let $\{X_i\}_{i\in I}$ be a family of well-filtered subspaces of a sober space $X$. Then $\bigcap_{i\in I}X_i$ is a well-filtered subspace.
\end{lem}
\begin{proof}
We only need to consider the case when $\bigcap_{i\in I}X_i\not=\emptyset.$

Let $\{K_i\}_{i\in I}$ be a filtered family of compact saturated subsets of the subspace $\bigcap_{i\in I}X_i$ and  $U$ be an open set of $X$ such that
$\bigcap_{i\in I}K_i\subseteq U$. If $K_i\not \subseteq U$ for all $i\in I$, then by Lemma \ref{subspace}, there is $e\not\in \bigcap_{i\in I}X_i$ such that
$$ \bigcap_{i\in I}\uparrow_X(\downarrow_X\!e\cap K_i)=\uparrow_Xe.$$
Thus there is $i_0$ such that  $e\not\in X_{i_0}$.  Note that $K_{i}\subseteq X_{i_0}$ for each $i\in I$.

$$ \bigcap_{i\in I}\uparrow_{X_{i_0}}(\downarrow_X\!e\cap K_i)=\bigcap_{i\in I}\uparrow_{X}\!\!(\downarrow_X\!e\cap K_i)\cap X_{i_0}=\uparrow_X\!\!e\cap X_{i_0}\subseteq X_{i_0}-\downarrow_X\!e.$$
Since $X_{i_0}$ is well-filtered, there is $i'\in I$ such that $\downarrow_X\!e\cap K_{i'}\subseteq X_{i_0}-\downarrow_X\!e$.
By the assumption, $\downarrow_X\!e\cap K_{i'}\not=\emptyset$. Choose $u\in \downarrow_X\!e\cap K_{i'}$. Then $u\in \downarrow_X\!e$, thus $u\not\in X_{i_0} - \downarrow_X\!e$. This contradicts $\downarrow_X\!e\cap K_{i'}\subseteq X_{i_0}-\downarrow_X\!e$. This contradiction shows that there must be $K_{i}$ such that $K_i\subseteq U$, hence $\bigcap_{i\in I}X_i$ is well-filtered.
\end{proof}

Now all conditions (K1)-(K4) are satisfied by the well-filteredness, therefore we have the following result.
\begin{thm}\label{WF-fication}
For any $T_0$ space $X$, there is a well-filtered space $W(X)$ and a continuous mapping $\eta_X: X\lra W(X)$ which is universal from $X$ to well-filtered spaces.
\end{thm}

\begin{cor}
The category of all well-filtered spaces is reflexive in the category of all $T_0$ spaces.
\end{cor}

\section{The product of two well-filtered spaces is well-filtered}

It is well-known that the product of two sober spaces is sober \cite{comp}. However it is still unknown whether the product of two well-filtered spaces is well-filtered.

\begin{prop}
If $X$ and $Y$ are well-filtered spaces, then the product space $X\times Y$ is well-filtered.
\end{prop}
\begin{proof}
Let $\{K_i\}_{i\in I}$ be a filtered family of compact saturated subsets of $X\times Y$ and $W\subseteq X\times Y$ open such that
$$\bigcap_{i\in I}K_i\subseteq U.$$
Assume that $K_i\cap U^c\not=\emptyset$ for all $i\in I$. Then there is a minimal closed set $F\subseteq X\times Y$, $F\subseteq U^c$ such that
$$F\cap K_i\not=\emptyset (\forall i\in I).$$
Then
$$\bigcap_{i\in I}\uparrow_X\!p_X(K_i\cap F)=\uparrow_XC_X, ~~~ \bigcap_{i\in I}\uparrow_X\!p_Y(K_i\cap F)=\uparrow_YC_Y,$$
where $C_X\subseteq X $ and $ C_Y\subseteq Y$ are nonempty anti-chains, and $p_X: X\times Y\lra X$ and $p_Y: X\times Y\lra Y$ are the projection mappings.

Choose an element $x_0\in C_X$ and an element  $y_0\in C_Y$. We have the following
$$(\downarrow_X\!x_0\times Y)\cap (F\cap K_i)\not=\emptyset (\forall i\in I).$$

In fact, for each $i\in I$, $x_0\in \uparrow_X\!p_X(K_i\cap F)$, so there exists $(u_1, u_2)\in K_i\cap F$ with $x_0\ge u_1$. Hence $(u_1, u_2)\in (\downarrow_X\!x_0\times Y)\cap (K_i\cap F)$.

Similarly, $(X\times \downarrow_Y\!y_0)\cap F\cap K_i\not=\emptyset (\forall i\in I)$.

By the minimality of $F$, we have $F\subseteq \downarrow_X\!x_0\times Y$, as well as $F\subseteq X\times \downarrow_Y\!y_0$. Therefore
$$ F\subseteq (\downarrow_X\!x_0\times Y)\cap (X\times \downarrow_Y\!y_0)=\downarrow_X\!x_0\times \downarrow_Y\!y_0.$$
Since $F\cap K_i\not=\emptyset$, $(\downarrow_X\!x_0\times \downarrow_Y\!y_0)\cap K_i\not=\emptyset$ holds for each $i\in I$.
Since each $K_i$ is  saturated , we have $(x_0, y_0)\in K_i (\forall i\in I)$. Thus $(x_0, y_0)\in \bigcap_{i\in I}K_i\subseteq U$.
There are open sets $U_1\subseteq X, U_2\subseteq Y$ such that $(x_0, y_0)\in U_1\times U_2\subseteq U$.

Applying Lemma \ref{fund lemma} to $\{p_X(K_i\cap F): i\in I\}$ and $\{p_Y(K_i\cap F): i\in I\}$ we have
$$\bigcap_{i\in I}\uparrow_X\!(\downarrow_X\!x_0\cap p_X(K_i\cap F))=\uparrow_X\!x_0, \bigcap_{i\in I}\uparrow_Y\!(\downarrow_Y\!y_0\cap p_Y(K_i\cap F))=\uparrow_Y\!y_0.$$
As $X$ and $Y$ are well-filtered, and $\{K_i: i\in I\}$ is filtered, there is a $K_{i_0}$ such that
$$\downarrow_X\!x_0\cap p_X(K_{i_0}\cap F)\subseteq U_1, \downarrow_Y\!y_0\cap p_Y(K_{i_0}\cap F)\subseteq U_2.$$
Thus
\[
\begin{array}{rcl}
F\cap K_{i_0}&\subseteq& (\downarrow_X\!x_0\times \downarrow_Y\!y_0)\cap (p_X(K_{i_0}\cap F)\times p_Y(K_{i_0}\cap F))\\
             &=&(\downarrow_X\!x_0\cap p_X(K_{i_0}\cap F))\times (\downarrow_Y\!y_0\cap p_Y(K_{i_0}\cap F)\\
             &\subseteq& U_1\times U_2\\
             &\subseteq& U.\end{array}\]
This contradicts $F\subseteq U^c$.

The proof is completed.

\end{proof}

\noindent{\bf Acknowledgement}
This work was supported by Singapore Ministry of Education Academic Research Fund Tier 2 grant MOE2016-T2-1-083 (M4020333); NTU Tier 1 grants  RG32/16 (M4011672); NSFC (11661057); the Ganpo
555 project for leading talent of Jiangxi Province and the Natural
Science Foundation of Jiangxi Province, China (20161BAB2061004); NSFC (11361028, 61300153, 11671008, 11701500, 11626207); NSF Project of Jiangsu Province, China (BK20170483); NIE AcRF (RI 3/16 ZDS), Singapore.

\end{document}